\documentclass[12pt]{article}

\usepackage{latexsym,amssymb,amsgen,amsmath,amsxtra, amsfonts, amscd, amsthm}

\usepackage{authblk}

\usepackage{mathrsfs}

\usepackage{todonotes}

\usepackage{color}

\usepackage{tikz}

\newtheorem{theorem}{Theorem}

\newtheorem{lemma}[theorem]{Lemma}

\newtheorem{remark}{Remark}

\newtheorem{definition}[theorem]{Definition}

\newcommand{\R}{{\mathbb{R}}}
\newcommand{\Z}{{\mathbb{Z}}}

\def\be{\begin{equation}}
\def\ee{\end{equation}}
\def\ba{\begin{array}}
\def\ea{\end{array}}
\def\vp{\varphi}

\newcommand{\spliced}[3]{#1\bowtie_{#2}\!#3}
\newcommand{\norm}[1]{\|{#1}\|}
\newcommand{\Dt}{[0, +\infty) \times D}
\newcommand{\U}{\mathfrak{u}}

\title{Semiflow selection and Markov selection theorems}
\author{Jorge E. Cardona and Lev Kapitanski}
\affil{
Department of Mathematics, University of Miami, Coral Gables, FL 33124, USA}

\date{}

\providecommand{\keywords}[1]{\textit{Keywords:} #1}

\begin{document}

\maketitle


\begin{abstract}
The deterministic analog of the Markov property of a time-homogeneous Markov process is the semigroup property
of solutions of an autonomous differential equation. The semigroup property arises naturally when
the solutions of a differential equation are unique, and leads to a semiflow. We prove an abstract
result on measurable selection of a semiflow for the situations without uniqueness.  We outline applications to ODEs, PDEs, differential inclusions, etc. Our proof of the semiflow selection
theorem is motivated by N. V. Krylov's Markov selection theorem. To accentuate this connection,
we include a new version of the Markov selection theorem related to more recent papers of
Flandoli \& Romito and Goldys et al.

\end{abstract}

\let\thefootnote\relax

\let\thefootnote\relax\footnote{\keywords{measurable selection theorems, semiflows, Markov property, Markov selection theorem}}


\section{Introduction}

This paper originates in our studies of the Markov property for statistical solutions of
evolution PDEs. There, in probabilistic setting, the fundamental result is the Markov selection
theorem of N. V. Krylov \cite{Krylov} and its connection with the martingale problem as shown by
D. W. Stroock and S. R. S. Varadhan \cite{S-V}. While reading these sources along with more recent developments  of F. Flandoli and M. Romito \cite{F-R} and B. Goldys et al \cite{G-R-Z}, we have realized that there is a deterministic result lurking behind probabilistic considerations. The deterministic analog of the Markov property of a time-homogeneous Markov process is the semigroup property
of solutions of an autonomous differential equation. The semigroup property arises naturally when
the solutions of a differential equation are unique, and leads to a semiflow.  When there is no uniqueness, the literally understood semigroup property does not make sense (except when lifted to dynamics in spaces of sets, or when, as proposed by G. Sell, dynamics is viewed on the space of trajectories \cite{Se}). However, as we have discovered
and prove in this paper,
there is a measurable {\it selection} with the semigroup property. In a vast literature on selection theorems (see the surveys \cite{Wa,Io} and a more recent account in \cite{H-P})  the semigroup question it seems has been overlooked. Our result remedies this situation.

Let $X$ be a set,  and let $\Omega$ be the space of all one-sided infinite paths in $X$ (i.e., the maps
$w: [0, +\infty)\to X$;
both continuous and discrete time cases are admissible, but assume $t\in [0, +\infty)$ for now).
Let $S = \left(S(x)\right)_{x\in X}$ be a family of subsets of $\Omega$ such that each
$S(x)$ is a subset of $\Omega_x = \{w\in\Omega :\;w(0) = x\}$.
Think of $S(x)$ as Kneser's integral funnel, the set of all possible solutions of some differential equation with the same
initial condition $x$ at time $t = 0$.  A {\it selection} of the family $S$ is a map ${\mathfrak u}: [0, +\infty)\times X\to X$ such that ${\mathfrak u}(\cdot, x)\in S(x)$ for every $x\in X$ (a selection picks a solution for every initial condition). We say that the selection ${\mathfrak u}$ has the {\it semigroup property}, or that ${\mathfrak u}$ is a {\it semiflow selection}, if
\be\label{sg}
{\mathfrak u}(t_2, {\mathfrak u}(t_1, x)) = {\mathfrak u}(t_2 + t_1, x), \quad\forall t_1, t_2 \ge 0\;\;\forall x\in X\,.
\ee
We give precise statements and a proof of this result in Section~2.

\medskip

Selections with the semi-group  property are not, in general, unique. How badly non-unique
it can be is shown in a remarkable fashion by A. Beck \cite{Be}. Nevertheless,
our selection produces a semiflow that is unique as a unique maximizer of a certain countable
family of functionals.
However, there is a lot of freedom in the choice of this family,
and the maximizers for different families may, in principle, differ.
In Section~\ref{sec:uniq} we give a particularly simple example
where the selection can be made explicitly, and show how the selection depends on the choice of
the functional family.

\medskip

That we talk about semi-groups and not groups here is important. There may be no selection with the group property: consider
this simple example:  ${\dot x} = 2\;\hbox{sign}(x)\;\sqrt{|x|}$.
\medskip

 We include in this paper results on the Markov selection as well.  Though at first glance the Markov selection and the semigroup selection appear to be quite different, they have a lot in common (as we see it).
{\bf The Markov selection  setting} can be described as follows.
Again, we have the set $X$, but now the action is happening in the space $\mathcal M$ of probability measures over the space of paths, $\Omega$.
Over every point $x\in X$ a set ${\mathscr C}(x)\subset {\mathcal M}$ is given, and from every
$P_x\in {\mathscr C}(x)$ issues a path $P_x[\theta_t^{-1}(\cdot) |\,{\cal F}_t](\cdot)$, $t\in [0, +\infty)$, of regular conditional probability distributions such that
$P_x[\theta_t^{-1}(\cdot) |\,{\cal F}_t](w)\in{\mathscr C}(w(t))$ for $P_x$-almost all $w\in\Omega$.  (Here $\theta_t$
is the shift on paths, $\theta_t w(\cdot) = w(t + \cdot)$, and $\left({\cal F}_t\right)$ is a filtration of
$\sigma$-algebras on $\Omega$.)
A selection of the family ${\mathscr C} = \{{\mathscr C}(x), x\in X\}$ is a map $P_{(\cdot)} : X\to {\mathcal M}$ such that
$P_x\in {\mathscr C}(x)$ for all $x\in X$. The selection $P_{(\cdot)}$ is Markov (has Markov property) if
$P_x[\theta_t^{-1}(\cdot) |\,{\cal F}_t](w)= P_{w(t)(\cdot)}$ for $P_x$-almost all $w\in\Omega$ and all $t > 0$.
In Section~3 we prove an abstract result on what we call Krylov set-valued maps. We hope to have simplified (a little bit) Krylov's argument. An old theorem of V.~Strassen \cite[Theorem 3]{Strassen} plays significant role in our considerations (not surprisingly, it is a selection theorem). The Markov selection and the strong Markov selection theorems (in the form close to that used in recent papers
\cite{F-R,G-R-Z}) are obtained as corollaries of abstract selection results.
\medskip

In Section~\ref{sec:examples} we discuss briefly applications of selections with the semigroup property to ODEs,
differential inclusions, and the three-dimensional Navier-Stokes equations. Applications of the Markov selection theorems will be discussed elsewhere.

\bigskip

{\sc \bf Acknowledgment:}
It was kindly brought to our attention by Professor E.~Feireisl that our application to the Leray-Hopf solutions of the Navier-Stokes equations required an extra analysis of the singular times where the strong energy inequality is not valid.
The results in Section~\ref{gener} allow us to treat such situations.


\section{ Semiflow selection}\label{sec:semi}

\subsection{Basic case} \label{sec:basic-case}

\noindent{\bf --}\ \ $X$ is a separable complete metric space (with metric $\rho_X$);  ${\cal B}_X$ is
the Borel $\sigma$-algebra of $X$; notation
${\cal B}_Z$ will be used for the Borel $\sigma$-algebras of various metric spaces $Z$.

\noindent{\bf --}\ \ $\Phi$ is a countable set of bounded continuous functions on $X$ that separate
points of $X$.

\noindent{\bf --}\ \ ${\cal T}$ (time) is $\R_+ = [0, +\infty)$ or $\Z_+ = \{0, 1, 2, \dots\}$;

\noindent{\bf --}\ \ $\Omega = C({\cal T}\to X)$ is the space of continuous (infinite) paths in $X$ with the topology of uniform convergence on compact subsets of ${\cal T}$. $\Omega$ is a separable complete metric space with the metric
\be\label{met-Om}
d(u, v) = \sum_{\ell = 1}^\infty 2^{-\ell}\,\sup_{t\in [0, \ell]}\rho(u(t), v(t))\;[1 + \sup_{t\in [0, \ell]}\rho(u(t), v(t))]^{-1}\,.
\ee
$\Omega_x$ denotes the set of all paths in $\Omega$ starting at the point $x$;

\noindent{\bf --}\ \ Given two paths $w$ and $v$ such that $w(s) = v(0)$, we define their splicing
$w \bowtie_s v$
as a new path
\[
w \bowtie_s v(t) = \begin{cases}
w(t) & \text{when}\;0\le t\le s \\
v(t-s) & \text{when}\; t \ge s\,.
\end{cases}
\]
\bigskip

\noindent{\bf --}\ \ $\theta_s$ is a shift on $\Omega$: $\theta_s : w(\cdot) \to w(s + \cdot)$;
\bigskip

The dynamical structure is given by the shifts $\theta_s$  and by a set-valued map ${S} : X \to 2^{\Omega}$ with the following properties.
\bigskip

\noindent{\bf Properties of ${S}$}.
\begin{enumerate}
\item[{\bf S1}]{\bf --}\  ${S}$ maps points $x\in X$ to compact subsets ${S}(x)\subset \Omega_x$
\item[{\bf S2}]{\bf --}\  The map ${S} : X \to 2^{\Omega}$ is measurable in the sense that for any closed set $K\subset \Omega$, the set
\[
S^-(K) = \{x\in X : {S}(x) \cap K \neq \emptyset\}
\]
is Borel in $X$. Assumptions {\bf S1} and {\bf S2} (and the assumptions on $X$) imply that the map ${S}$ satisfies the conditions of the measurable selection theorem of Kuratowski and Ryll-Nardzewski,
\cite[Theorem 1]{K-R}, and therefore, there exists
a measurable selection of ${S}$, i.e., a $({\cal B}_X, {\cal B}_\Omega)$-measurable map
${\mathfrak u}: X\to \Omega$ such that ${\mathfrak u}(x)\in {S}(x)$ for all $x\in X$.
\item[{\bf S3}]{\bf --}\  The map ${S}$ is compatible with the semigroup $\theta_s$ in the sense that
\be\label{compat-1}
w\in {S}(x) \implies \theta_s(w)\in {S}(w(s))\,,\quad \forall s\in{\cal T}\,,\;\forall x\in X\,.
\ee
\item[{\bf S4}]{\bf --}\ If $w\in {S}(x)$ and $v\in {S}(w(s))$ for some $s\in {\cal T}$, then
$w\bowtie_s v\in {S}(x)$. In particular, if
${\mathfrak u}$ is any measurable selection of ${S}$,
then the path $w \bowtie_s {\mathfrak u}(w(s))$ belongs to ${S}(x)$.
\end{enumerate}

\noindent{\bf --}\ \ Selections of ${S}$ are paths, so ${\mathfrak u}(x)$ denotes a path in ${S}(x)$. When we want to show the location of the path ${\mathfrak u}(x)$ at time $t$, we write
${\mathfrak u}(t, x)$. In particular, ${\mathfrak u}(0, x) = x$.

\begin{theorem}\label{thm:main}
The set-valued map ${S}$ has a measurable selection ${\mathfrak u}$ with the semigroup property:
\be\label{sg-4}
{\mathfrak u}(t_2, {\mathfrak u}(t_1, x)) = {\mathfrak u}(t_1 + t_2, x)\,,
\ee
for all $t_1, t_2\in {\cal T}$, and for all $x\in X$. The corresponding semigroup $U(t): X\to X$,
defined by the formula $U(t)(x) = {\mathfrak u}(t, x)$ for all $t\ge 0$ and all $x\in X$, is a
Borel measurable map for every $t$.
\end{theorem}
\begin{proof}  There are three steps in the proof.


Step 1. Recall that $\Phi$ is a countable subset of the space $C_b(X)$  (of bounded continuous functions on $X$) that separates points of $X$. Choose a
countable dense subset  $\Lambda$ of $(0, +\infty)$, and let $(\lambda_n, \vp_n)$, $n = 1, 2, \dots$, be some enumeration of the Cartesian product $\Lambda\times \Phi$.
Associate with each  $(\lambda_n, \vp_n)$ the following continuous function $\zeta_n$ on $\Omega$:
\be\label{zeta_n}
\zeta_n(w) = \int_0^\infty e^{-\lambda_n t}\,\vp_n(w(t))\,dt
\ee
(in the case ${\cal T}$ is discrete, we replace the integral by a sum). Let $\zeta$ be one of the functions $\zeta_n$,
\[
\zeta(w) = \int_0^\infty e^{-\lambda t}\,\vp(w(t))\,dt\,.
\]
$\zeta$ is a continuous function on $\Omega$, and as such, it attains its maximum on each compact set
${S}(x)$. Denote
\be\label{V}
V_{\zeta}[{S}(x)] = \{ w\in {S}(x) : \zeta(w) = \max_{v\in {S}(x)} \zeta(v)\}\,.
\ee
We will say that $w$ is a maximizer of $\zeta$ in ${S}(x)$ if $w\in V_{\zeta}[{S}(x)]$. By extension,
$V_{\zeta}[{S}]$ will denote the set-valued map $x\mapsto V_{\zeta}[{S}(x)]$.

It turns out that the new set-valued map $V_{\zeta}[{S}] : X \to 2^{\Omega}$ has the same properties
{\bf S1} through {\bf S4} as the map ${S}$.

Indeed, the sets $V_{\zeta}[{S}(x)]$ are compact since the sets ${S}(x)$ are compact and $\zeta$ is continuous. Measurability of $V_{\zeta}[{S}]$ follows from the following version of the Dubins and Savage (see \cite{Dubins-Savage}) measurable maximum theorem.
\begin{theorem}\label{thm:meas-max}
Suppose  $S$ is a weakly-measurable set-valued map from a measurable space
$(X, \Sigma)$  with compact values in a separable metrizable space $Y$. Let $f: X\times Y\to \R$
be a Crath\'eodory function (i.e., $f(\cdot, y)$ is $(\Sigma, {\cal B}_{\R})$-measurable for every $y\in Y$ and $f(x, \cdot)$ is continuous for every $x\in X$). Then the set-valued map
\[
V_f[S]: x \mapsto \{y\in S(x): f(x, y) = \max_{z\in S(x)} f(x, z)\}\,
\]
(has non-empty compact values and) is measurable.
\end{theorem}
For the proof see \cite[Theorem 18.19]{A-B}. The notions of a weakly-measurable and a measurable
(in the sense of property {\bf S2})
set-valued map are equivalent in the case the target space is separable metrizable and $S(x)$ is non-empty compact for every $x$, as follows from \cite[Theorem 18.10]{A-B}.

Thus, the map $V_{\zeta}[{S}]$ is measurable. To establish {\bf S3}, pick a path
$w\in V_{\zeta}[{S}(x)]$. We need to show that $\theta_s(w)$ is a maximizer in ${S}(w(s))$.
Let $v$ be any path in ${S}(w(s))$. Then we know that the spliced path $w\bowtie_s v$ is in ${S}(x)$.
Since $w$ is a maximizer in ${S}(x)$, $\zeta(w) \ge \zeta(w\bowtie_s v)$, which implies
\[
\int_s^\infty e^{-\lambda t} \vp(w(t))\,dt \ge \int_s^\infty e^{-\lambda t} f\vp(v(t - s))\,dt =
e^{- \lambda s}\,\int_0^\infty e^{-\lambda t} \vp(v(t))\,dt = e^{- \lambda s}\,\zeta(v)\,.
\]
Since
\[
\zeta(\theta_s(w)) = \int_0^\infty e^{-\lambda t} \vp(w(t + s))\,dt = e^{\lambda s}\,\int_s^\infty e^{-\lambda t} \vp(w(t))\,dt\,,
\]
we have
\[
\zeta(\theta_s(w)) \ge \zeta(v)\,.
\]
This is true for all $v\in {S}(w(s))$. Hence, $\theta_s(w)$ is a maximizer in ${S}(w(s))$.

It remains to check property {\bf S4} for $V_{\zeta}[{S}]$. Let $w$ be a maximizer in ${S}(x)$
and let $v$ be a maximizer in ${S}(w(s))$. Consider the spliced path $w\bowtie_s v$. Then
\[
\begin{aligned}
& \zeta(w\bowtie_s v) = \int_0^s e^{-\lambda t} \vp(w(t))\,dt + \int_s^\infty e^{-\lambda t} \vp(v(t - s))\,dt = \\
& \int_0^s e^{-\lambda t} \vp(w(t))\,dt + e^{- \lambda s}\,\zeta(v) = \int_0^s e^{-\lambda t} \vp(w(t))\,dt + e^{- \lambda s}\,\zeta(\theta_s(w)) = \zeta(w)\,.
\end{aligned}
\]
Since $w$ is assumed to be a maximizer in ${S}(x)$, the path $w\bowtie_s v$ must be a maximizer as well.
Note that $\zeta(v) = \zeta(\theta_s(w))$ because $\theta_s(w)$ is a maximizer as was shown above.
\medskip

Step 2. Define recursively ${S}^0 = {S}$ and ${S}^{n+1} = V_{\zeta_{n+1}}[{S}^n]$.
For every $x\in X$,  ${S}^{n}(x)$ is a monotone decreasing sequence of embedded non-empty compacta.  Hence, the intersection is not empty, and we denote it by ${S}^{\infty}(x)$,
\[
{S}^{\infty}(x) = \bigcap_{n=0}^\infty {S}^{n}(x)\,.
\]
Notice that the set-valued map ${S}^{\infty}$ enjoys the same properties {\bf S1} through {\bf S4} as ${S}$.
In particular, it is measurable as the intersection of measurable set-valued maps, see \cite[Lemma 18.4]{A-B}.  If $w, v\in {S}^{\infty}(x)$, then, for every $\vp\in\Phi$,
\[
\int_0^\infty e^{-\lambda t} \vp(w(t))\,dt = \int_0^\infty e^{-\lambda t} \vp(v(t))\,dt
\]
for all $\lambda$ in the dense in $[0,+\infty)$ set $\Lambda$.
Thanks to the uniqueness of the Laplace transform, then $\vp(w(t)) = \vp(v(t))$ for all $t\ge 0$.
Since the set $\Phi$ separates points in $X$, $w(t) = v(t)$ for all $t$. Thus,
${S}^{\infty}(x)$ is a singleton.
\medskip

Step 3. Define the selection ${\mathfrak u}$ that assigns to every $x$ the unique path in ${S}^{\infty}(x)$.
This is the only possible selection of ${S}^{\infty}$ and it must be measurable by the Kuratowski-- Ryll-Nardzewski theorem. The property {\bf S3} enjoyed by ${S}^{\infty}$ ensures that
$\theta_s{\mathfrak u}(\cdot, x)$ lies in ${S}^{\infty}({\mathfrak u}(s, x))$, which proves the semigroup property
${\mathfrak u}(\cdot + s, x) = {\mathfrak u}(\cdot, {\mathfrak u}(s, x))$ for all $s$ and $x$.

Now consider the maps $U(t): X\owns x\mapsto {\mathfrak u}(t, x)\in X$. Clearly, they form a semigroup:
$U(0)$ is the identity map and $U(t_2)\circ U(t_1) = U(t_2+t_1)$.
Each map $U(t)$ is a composition of the $({\cal B}_X, {\cal B}_\Omega)$-measurable selection ${\mathfrak u} : X\to \Omega$ and the evaluation map
$\pi_{t} : \Omega\owns w(\cdot)\mapsto w(t)\in X$, which is continuous for every $t$.  Consequently, $U(t) : X\to X$ is Borel measurable. This
completes the proof of the theorem.

\end{proof}


\subsection{A generalization}\label{gener}

For applications to PDEs, it may be necessary to weaken somewhat assumptions of Theorem~\ref{thm:main} (see the Navier-Stokes example in Section~\ref{sec:NS}).
What follows are not ultimate, but are sufficient generalizations for our current purposes.

\subsubsection{Conditional shifts}\label{sec:conditional-shifts}

Consider the same situation as in Section~\ref{sec:basic-case}, but assume now that the set-valued map $S$ satisfies only the properties \textbf{S1}, \textbf{S2} and
\begin{enumerate}
\item[{\bf S3a}]{\bf --}\
  For every $x \in X$, $w \in S(x)$, and $s \in \mathcal{T}$, $$\theta_s(w) \in S(w(s)) \quad \implies \quad \left\{\spliced{w}{s}{u} : u \in S(w(s)) \right\} \subset S(x)\,.$$
\end{enumerate}
It is worth mentioning that the property \textbf{S3a} can be interpreted as the property \textbf{S4} conditioned to the cases when \textbf{S3} is satisfied.

\begin{theorem}\label{thm:conditional-selection}
  The set-valued map $S$ satisfying the properties \textbf{S1}, \textbf{S2}, and \textbf{S3a} has a measurable selection $\U$ with the following semigroup property: for every $x \in X$ and $t_1 \in \mathcal{T}$, if $\U(t_1 + \cdot, x) \in S(\U(t_1))$ then $$\U(t_2, \U(t_1, x)) = \U(t_1 + t_2, x)$$ for all $t_2 \in \mathcal{T}$.
\end{theorem}

\begin{proof}
  The proof is similar to that of Theorem~\ref{thm:main}.

  Step 1. Let $\{(\lambda_n, \phi_n)\}$ be some enumeration of $\Lambda \times \Phi$ where $\Lambda$ is a dense and countable subset of $(0, +\infty)$ and $\Phi$ is a countable family of separating functions of $X$.
  Construct the set-valued maps $V_\zeta[S]$ as in (\ref{V}), properties \textbf{S1} and \textbf{S2} will be preserved as before.

  The property \textbf{S3a} is also preserved.
  Indeed, let $w$ be a maximizer in $S(x)$ with $\theta_s(w)$ a maximizer in $S(w(s))$.
  Consider the spliced path $\spliced{w}{s}{v} \in S(x)$, where $v$ is also a maximizer in $S(w(s))$, thus $\zeta(\theta_s(w)) = \zeta(v)$ and
  \begin{align*}
    \zeta(\spliced{w}{s}{v}) &= \int_0^s e^{-\lambda t} \vp(w(t))\,dt + e^{- \lambda s}\,\zeta(v) \\ &= \int_0^s e^{-\lambda t} \vp(w(t))\,dt + e^{- \lambda s}\,\zeta(\theta_s(w)) = \zeta(w)\,.
  \end{align*}
  Since $w$ is assumed to be a maximizer in ${S}(x)$, the path $\spliced{w}{s}{v}$ must be a maximizer as well.

  Step 2. Define recursively the maps $S^0 = S$ and $S^{n+1} = V_{\zeta_{n+1}}[S^n]$, as well as $S^\infty(x) = \cap_{n=0}^\infty S^n(x)$. As in the proof of Theorem~\ref{thm:main}, $S^\infty$ takes values in singleton sets, and every $S^n$ and $S^\infty$ satisfy the properties {\bf S1}, {\bf S2}, and {\bf S3a}.

  Step 3. Define $\U$ as the selection that assigns to $x$ the unique element in $S^\infty(x)$, this map is measurable in the usual sense.

  To establish the semigroup property, fix $x \in X$, $t_1 \in \mathcal{T}$ and assume that $\U(t_1 + \cdot, x) \in S(\U(t_1,x))$. Notice that $\U(x) \in S^n(x)$ for every $n \geq 0$, and take $w \in S^1(\U(t_1,x))$. Since the map $S^1$ satisfies the property {\bf S3a} the splice $\spliced{\U(x)}{t_1}{w}$ is in $S^1(x)$, which means that $\zeta_1(\U(x)) = \zeta_1(\spliced{\U(x)}{t_1}{w})$, hence $\zeta_1(\U(t_1 + \cdot, x)) = \zeta_1(w)$ and $\U(t_1 + \cdot, x) \in S^1(\U(t_1, x))$. Take $w \in S^2(\U(t_1,x))$ and continue in the same way to obtain that $\U(t_1 + \cdot, x) \in S^n(\U(t_1, x))$ for very $n \geq 0$, hence $\U(t_1 + \cdot, x) \in S^\infty(\U(t_1, x))$. This completes the proof.
\end{proof}

\subsubsection{Inductive limit}\label{sec:inductive-limit}

Let $X^1\subset X^2\subset X^3\subset \cdots$ be an increasing sequence of topological $T_1$-spaces.  Assume
that the inclusion maps $\phi_i: X^i\to X^{i+1}$ are homeomorphisms between $X^i$ and $\phi_i(X^i)\subset X^{i+1}$ (in particular, each $\phi_i$ is and open and closed map, and $\phi_i(X^i)$ is closed in $X^{i+1}$).
Define $X$ as the direct (inductive) limit of spaces $X^n$ with maps $\phi_n$:
\[
X = \varinjlim X^n \,.
\]
As a set, $X = \cup_n X^n$. We equip $X$ with the final topology, i.e., each inclusion map $f_i : X^i\to X$ is continuous. A set $U\subset X$ is open (closed) iff $f_i^{-1}(U) = \{x_i\in X^i : f_i(x_i)\in U\}$ is open (closed) in $X^i$ for every $i$. The limit space $X$ is $T_1$ (because, for every point $x\in X$, $f_i^{-1}(x) = x$ or $f_i^{-1}(x) = \emptyset$, in any case a closed set). Every finite subset of $X$ is closed and compact. Each $X^i$ is a closed subset of $X$. It turns out that for every compact set $K\subset X$ there is an $i_0$ such that
$K\subset f_i(X^i) = X^i$ for all $i\ge i_0$. This is a version of Steenrod's lemma \cite[Lemma 9.3]{Steenrod}.

\medskip

Next, consider the spaces of paths $\Omega^i = C_{loc}([0, +\infty)\to X^i)$ with compact-open topologies. We have $\Omega^i\subset\Omega^j$ continuously when $j \ge i$, and we define the inductive limit
\[
\Omega = \varinjlim\; \Omega^i
\]
as before. The maps $f_i$ map paths $w(\cdot)$ in $\Omega^i$ to continuous paths $f_i(w(\cdot))$ in $\Omega$. As in the previous section, $\Omega_x$ denotes all paths in $\Omega$ that start at $x\in X$.
If $x\in X$ and $i_0$ is the smallest integer such that $x\in X^{i_0}$, then
$\Omega_x = \cup_{i\ge i_0}\Omega_x^{i}$.
Note, that for any compact subset of $\Omega$, there is an $i$ such that this compact subset lies
entirely inside $\Omega^{i}$.
\medskip

We will consider set-valued maps $S: X\to 2^\Omega$. It is important to specify what measurability
of such maps means.
\medskip

Assume each space $X^i$ carries a $\sigma$-algebra ${\cal B}^i$ of its subsets. Then we can defined
a $\sigma$-algebra ${\cal B}_X$ for $X$ as follows: a set $B$ belongs to ${\cal B}_X$ iff
$f_i^{-1}(B)\in {\cal B}^i$ for every $i$.
[Check: the complements are in ${\cal B}_X$ because $f_i^{-1}(X\setminus B) = X^i\setminus f_i^{-1}(B)$.
The countable unions are in ${\cal B}_X$ because $f_i^{-1}(\cup_n A_n) = \cup_n f_i^{-1}(A_n)$.]
\medskip

Once the $\sigma$-algebra ${\cal B}_X$ is specified, we say that a set-valued map $S: X\to 2^\Omega$
(with values in the non-empty closed subsets of $\Omega$) is closed-measurable if
\be\label{meas}
S^{-}(C) = \{x\in X: S(x) \cap C \neq \emptyset\}\in {\cal B}_X
\ee
for all closed $C\subset\Omega$.
\medskip

With these preliminaries, we can state the result.
\begin{theorem}\label{thm:main-2}
  Assume that each space $X^i$ is Polish and that ${\mathcal B}^i$ is its Borel $\sigma$-algebra.
  Let $S$ be a set-valued map $S: X\to 2^{\Omega}$ with the properties {\bf S1 -- S4} (measurability in property {\bf S2} is understood as in \eqref{meas}). In addition, assume that for every $i$ there is an integer $k_i\ge i$ such that $S(X^i)\subset \Omega^{k_i}$.
  Then $S$ has a measurable selection with the semigroup property.
\end{theorem}
\begin{proof}
The proof is essentially the same as that of Theorem~\ref{thm:main}. There are just a few modifications
needed. For $i =1, 2, \dots$, extend every function in $\Phi^i$ to a function on $X$ using the Tietze-Urysohn extension  theorem (going from $X^i$ to $X^{i+1}$ to $X^{i+2}$ etc.) and keeping the bound. The set of extended functions will be called again $\Phi^i$. Define $\Phi$ as the union of all the sets $\Phi^i$.  The functions from $\Phi$ are continuous and bounded on every $X^i$ and separate the points in  $X$. Let $\{(\lambda_n, \vp_n)\}$ be some enumeration of $\Lambda\times\Phi$, where $\Lambda$ is a dense countable subset of $(0, +\infty)$, and define the functionals $\zeta_n$ as in  \eqref{zeta_n}.
Then proceed with the construction of the set-valued map $V_{\zeta}[S]$ as in \eqref{V}.

Note, that for every set $S(x)$, since it is compact in $\Omega$, there is an $i$ such that $S(x)\subset \Omega^i_x$ (and $x\in X^i$, of course).  Hence, $V_{\zeta}[S(x)]\subset \Omega^i_x\subset \Omega_x$.
The properties {\bf S3} and {\bf S4} for $V_{\zeta}[S]$ are verified exactly as in the proof of Theorem~\ref{thm:main}. It remains to check measurability of $V_{\zeta}[S]$. That $S$ is measurable means
\[
\{x\in X: S(x) \cap C \neq\emptyset\}\in {\cal B}_X\,.
\]
Since
\[
\{x\in X: S(x) \cap C \neq\emptyset\} = \cup_i \{x\in X^i: S(x) \cap C \neq\emptyset\}\,,
\]
measurability of $S$ is equivalent to
\[
\{x\in X^i: S(x) \cap C \neq\emptyset\}\in {\cal B}^i\,.
\]
Due to the assumption $S(X^i)\subset \Omega^{k_i}$, we have
\[
\{x\in X^i: S(x) \cap C \neq\emptyset\} = \{x\in X^i: S(x) \cap C \cap \Omega^{k_i}\neq\emptyset\}\,.
\]
Thus, when restricted to $X^i$, the map $S: X^i\to 2^{\Omega^{k_i}}$ is measurable. It follows that
\[
\{x\in X^i: S(x) \cap C \cap \Omega^j\neq\emptyset\} \in {\cal B}^i\,
\]
for any $j$ (for $j < k_i$ use the fact that $C \cap \Omega^j$ is closed in $\Omega^{k_i}$, and for $j > k_i$ use the fact that $C \cap \Omega^j = C \cap \Omega^{k_i}$).
By the measurable maximum Theorem~\ref{thm:meas-max}, $V_\zeta[S]$ viewed as a set-valued map from $X^i$ into compact subsets of $\Omega^{k_i}$ is measurable. Arguing as before, we obtain
\[
\{x\in X^i: V_\zeta[S](x) \cap C \cap \Omega^j\neq\emptyset\} \in {\cal B}^i\,
\]
for every closed set $C\subset X$ and all $i, j \ge 1$, which proves measurability of $V_\zeta[S]$ as a map $X\to 2^\Omega$.

The remaining steps 2 and 3 go as in the proof of Theorem~\ref{thm:main}.

\end{proof}

\bigskip

A representative example of the above construction is as follows. Let $V$ be a reflexive separable Banach space.
Let $X^i$ be a closed ball in $V$ of radius $i$ centered at $0$ and equipped with the weak topology of $V$.
Then, as a set, $X$ is $V$. However, the topology on $X$ is finer than the weak topology of $V$.
In fact, it is the bounded weak topology on $V$. By the Banach-Dieudonn\'e theorem, \cite[Lemma V.5.4]{D-S}, a fundamental system of neighborhoods of the origin of $X$ consists of the sets $\{v\in V: |\langle v^*, v\rangle |< 1, \;v^*\in A^*\}$, where $A^* = \{v^*_n\}$ is a sequence in the dual space $V^*$ converging to zero.
A corollary of this is the fact that $X$ is a locally convex linear topological space.
By Proposition III of \cite{Heisey}, $X$ is regular and Lindel\"of (hence, paracompact), but not first countable (and not metrizable) if $V$ is infinite-dimensional.

The two previous generalization can be combined into the following theorem.

\begin{theorem}\label{thm:conditional-and-inductive}
  Assume that $X$ is an inductive limit of the Polish spaces $X^i$ as in Theorem~\ref{thm:main-2}, and let $S: X \to 2^\Omega$ be a set-valued map satisfying the properties \textbf{S1}, \textbf{S2} and \textbf{S3a}.
  In addition, assume that for every $i$ there is an integer $k_i\ge i$ such that $S(X^i)\subset \Omega^{k_i}$.
  Then, $S$ has a measurable selection with the semigroup property as in Theorem~\ref{thm:conditional-selection}, i.e., $$\U(t_2, \U(t_1, x)) = \U(t_1 + t_2, x)$$ for all $x \in X$, and for all $t_1, t_2 \in \mathcal{T}$, as long as $\U(t_1 + \cdot, x) \in S(\U(t_1))$.
\end{theorem}

\section{Markov selection}\label{sec:Mark}

\subsection{Notation}

\begin{enumerate}
\item[--\ ] $X$ is a Polish space, i.e., $X$ is Hausdorff, separable, and completely metrizable.
\item[--\ ] $\Phi$ is a countable family of bounded continuous functions on $X$ that strongly separates points of $X$ and is closed under multiplication.
\item[--\ ] $\Omega = C([0,+\infty)\to X)$ is the space of continuous, one-sided infinite paths in $X$;   $\Omega$ is equipped with the compact-open topology; it  is separable and metrizable.
\item[--\ ] $\Omega_x = \{w\in\Omega: w(0) = x\}$
\item[--\ ] ${\cal F}$ is the Borel $\sigma$-algebra on $\Omega$.
\item[--\ ] $({\cal F}_t)_{t\ge 0}$ is the (increasing) natural filtration of the measurable space
$(\Omega, {\cal F})$, i.e.,  ${\cal F}_t$ is the smallest $\sigma$-algebra that contains all the sets
of the form
\be\label{basic-sets}
\{w\in \Omega: w(t_1)\in A_1, \dots, w(t_n)\in A_n,\;\text{where}\;\le t_1 < \dots \le t_n\le t,\; A_1,\dots, A_n\in {\cal B}_X\}\,.
\ee
Note that ${\cal F} = {\cal F}_\infty = \sigma\left(\cup_t {\cal F}_t\right)$.
\item[--\ ] $\pi_s : \Omega\to X$ is the standard coordinate map $\pi_s(w) = w(s)$.
\item[--\ ] $\theta_s :\Omega\to\Omega$ is a shift: $(\theta_s w)(\cdot) = w(s + \cdot)$, for $s\ge 0$.
\item[--\ ] $C_b(\Omega)$ is the space of bounded continuous functions on $\Omega$.
\item[--\ ] ${\mathscr P}$ is the space of probability measures on $(\Omega, {\cal F})$.
${\mathscr P}$ is a compact convex set in the space of Borel measures on $\Omega$ with weak$^*$ topology generated by the functionals $P\mapsto P f = \int f(w)\,P(dw)$ with functions $f$ running through $C_b(\Omega)$.
\item[--\ ] The pairing between a measure $Q$ and a functions $f$ on $\Omega$ will be denoted $Q f$ or, in extended form,  $\int f(w)\,Q(dw)$.
\item[--\ ] $\theta_s P$ is the push-forward of the measure $P$ with the map $\theta_s$.
\item[--\ ] $P[\;\cdot\; | {\cal F}_s]$ is the regular conditional probability distribution of $P$ given the $\sigma$-algebra ${\cal F}_s$.
\item[--\ ] Throughout, $P$-a.s. is synonymous to ``for $P$-almost all $w$."
\item[--\ ] Throughout, $\int$ means $\int_\Omega$.
\end{enumerate}

The following lemma defines $P \otimes_s Q$, a standard ingredient in the description of Markov selections, \cite[Theorem 6.1.2]{S-V}.
\begin{lemma}\label{PxQ}
Let $Q_{\bullet}: w \mapsto Q_w \in \mathscr{P}$ be an $\mathcal{F}_s$-measurable map (i.e., for any $f \in C_b(\Omega)$ the map $w \mapsto Q_w f $ is $\mathcal{F}_s$-measurable).  Assume that for every $w$ the probability measure $Q_w$ is supported on the set $\Omega_{w(s)}$.
Then, for any measure $P \in \mathscr{P}$, there is a unique measure, denoted by $P \otimes_s Q$, such that $P \otimes_s Q$ and $P$ agree on $\mathcal{F}_s$, and $P[\theta_s^{-1}(\cdot) | \mathcal{F}_s](w) = Q_w $ for $P\otimes_s Q$-almost all $w\in\Omega$.
\end{lemma}

\subsection{Statement of the Markov selection theorem}\label{sub:statement}

We start with the following Markov selection theorem.
\begin{theorem}\label{thm:Markov}
  Let $\mathscr C$ be a set-valued map ${\mathscr C}: X\to 2^{\mathscr P}$
such that
\begin{description}
\item[C1] ${\mathscr C}(x)$ is a non-empty compact convex subset of ${\mathscr P}$ for every $x\in X$;
\item[C2] the map ${\mathscr C}$ is measurable in the sense that
\[
\{x\in X: {\mathscr C}(x)\cap K \neq\emptyset\}\in {\cal B}_X
\]
for every closed subset $K\subset {\mathscr P}$.

\item[C3] $P(\Omega_x) = 1$ for every measure $P$ in ${\mathscr C}(x)$, for every $x\in X$.
\end{description}
In addition, assume that $\mathscr{C}$ satisfies the following conditions:
  \begin{description}
  \item[MP1] For $x \in X$, if $P \in \mathscr{C}(x)$ and $s \geq 0$, then $$P[\theta_s^{-1}(\cdot) | \mathcal{F}_s](w) \in \mathscr{C}(w(s)) \quad (P,{\cal F}_s)\text{-a.s.}\,.$$
  \item[MP2] If $s \geq 0$, and if $Q$ is an $\mathcal{F}_s$-measurable selection of $w \mapsto \mathscr{C}(w(s))$, then $$P \otimes_s Q \in \mathscr{C}(x)\quad\text{provided}\; P\in {\mathscr C}(x).$$
  \end{description}
Then, there exists a measurable selection $P_x \in \mathscr{C}(x)$, $x\in X$,  such that  for any $x$, and $s \geq 0$, $$P_x[\theta_s^{-1}(\cdot) | \mathcal{F}_s](w) = P_{w(s)} \quad (P_x,{\cal F}_s)\text{-a.s.}\,.$$
\end{theorem}

This statement was inspired by the presentation of F. Flandoli and M. Romito in \cite[Section 2]{F-R}, where they prove this theorem (though their assumptions are somewhat different and more restrictive). We prove Theorem~\ref{thm:Markov} in Section~\ref{proof:Markov} by showing that it is a corollary  of a more general abstract selection theorem.
Our abstract theorem is close in spirit to Krylov's treatment in \cite{Krylov}.

In the statement of the theorem and below we write $(P,{\cal F}_s)\text{-a.s.}$ to mean
``outside of (possibly) a $P$-null set in ${\cal F}_s$."


\subsection{Formulation of an abstract selection theorem}

Let ${\mathscr C}$ will be our set-valued map with properties {\bf C1 - C3}.
As with most selection theorems, the proof of theorem \ref{thm:Markov} proceeds
with successive {\it reductions} of ${\mathscr C}$ to set-valued maps $X\to 2^{\mathscr P}$ with smaller, nested images that in the limit will become singletons and the resulting single-valued map will have the desired properties, this will be our selection. Reductions will be achieved by selecting measures that maximize certain linear functionals. It is convenient to formalize this process.
\bigskip

\noindent$\bullet$\ {\bf Markov kernels and Strassen's theorem.}

Let $(\Omega, \Xi)$ and $(R, \Sigma)$ be  measurable spaces. Recall, that a Markov kernel from $\Omega$ to $R$ is a real function ${\mathfrak  m}$ on $\Sigma\times \Omega$ such that for any $w\in \Omega$,
${\mathfrak m}(\cdot, w)$ is a probability measure on $\Sigma$, and for any $A\in\Sigma$, the map
$w\mapsto {\mathfrak m}(A, w)$ is $\Xi$-measurable. If $P$ is a probability measure on $\Xi$, then, by definition,
${\mathfrak m}\cdot P$ is a probability measure on $\Sigma$ such that
\[
({\mathfrak m}\cdot P)\, (A) = \int {\mathfrak m}(A, w)\,P(dw)\,.
\]
Assume now that the space $R$ is Polish and that $\Sigma$ is its Borel $\sigma$-algebra. The following result of V. Strassen, \cite[Theorem 3]{Strassen}, will be very useful to us.
\begin{theorem}[Strassen]
Let $C : w\mapsto C_w$ be a set-valued map on $\Omega$ with values in nonempty convex compact sets of
Borel probability measures on $R$. Assume that this map is measurable in the sense that the corresponding map from $\Omega$ to the support functions $H_w$ of the sets $C_w$ is $\Xi$-measurable.
Let $P$ be a probability measure on $\Omega$ and let $Q$ be a probability measure on $R$.
In order that there exist a Markov kernel ${\mathfrak m}$ from $\Omega$ to $R$ such that
\[
{\mathfrak m}\cdot P = Q
\]
and
\[
{\mathfrak m}(\,\cdot\,, w)\in C_w\quad\text{for $P$-almost all $w$},
\]
it is necessary and sufficient that
\[
\int_R g(r)\,Q(dr) \le \int_\Omega H_w[g]\,P(dw)
\]
for all $g\in C_b(R)$.
\end{theorem}

\noindent$\bullet$\ {\bf Maximization operations $V_\eta$.}\

Let $\eta$ be a linear continuous functional on ${\mathscr P}$ and let $K$ be a compact convex subset of ${\mathscr P}$.
Denote by $V_\eta[K]$ the following compact convex subset of $K$:
\be
V_\eta[K] = \{P\in K: \eta(P) = \sup_{Q\in K} \eta(Q)\}\,.
\ee
By extension, $V_\eta[{\mathscr C}]$ will denote the set-valued map $x\mapsto V_\eta[{\mathscr C}(x)]$.
\medskip

\noindent$\bullet$\ {\bf Set-valued maps $K(P, s, {\mathscr C})$.}

Denote by $h_x$ the support functions of the sets ${\mathscr C}(x)$. By definition,
\[
h_x[f] = \sup\,\{P f\,,\;P\in {\mathscr C}(x)\}
\]
for any $f\in C_b(\Omega)$. As a support function of a compact convex set,
$h_x$ is continuous, convex, and degree one positive homogeneous.
Measurability of ${\mathscr C}$ implies that, for every $f\in C_b(\Omega)$, the map $w\to h_{w(s)}[f]$ is $\sigma(\pi_s)$-measurable.
For a probability measure $P\in {\mathscr P}$ and a number $s\ge 0$,  define the set
\be\label{eq:K}
K(P, s, {\mathscr C}) = \{Q\in {\mathscr P}:\; Q f \le \int h_{w(s)}[f]\,P(dw) \,,\;\forall f\in C_b(\Omega, {\cal F}_s)\}\,.
\ee
Note that $\int h_{w(s)}[\cdot]\,P(dw)$ is the support function of $K(P, s, {\mathscr C})$.
This non-empty, convex, compact set depends on ${\mathscr C}$ through the support function $h_{w(s)}$ of ${\mathscr C}(w(s))$.
By Strassen's theorem, for every $Q\in K(P, s, {\mathscr C})$ there exists a
Markov kernel $Q_w(dv)$ such that $Q_w\in {\mathscr C}(w(s))$ for $P$-almost all $w$,
and $Q = Q_w\cdot P$, i.e.,
\be
Q f = \int \left(Q_w f\right)\,P(dw)
\ee
for every function $f\in C_b(\Omega, {\cal F}_s)$.


The following lemma shows that the correspondence ${\mathscr C}\to K(P, s, {\mathscr C})$ commutes with the maximization operators $V_\eta$.
\begin{lemma}\label{commute}
For any linear continuous functional $\eta$ on ${\mathscr P}$,
\be
V_\eta[K(P, s, {\mathscr C})] = K(P, s, V_\eta[{\mathscr C}])\,.
\ee
\end{lemma}
\begin{proof} If $Q\in K(P, s, V_\eta[{\mathscr C}])$, there is a Markov kernel $Q_w$ such that
$Q_w\in V_\eta[{\mathscr C}(w(s))]$ $P$-a.s., and $Q_w$ is an $\eta$-maximizer in ${\mathscr C}(w(s))$. If $Q^\prime\in V_\eta[K(P, s, {\mathscr C})]$, then
$Q^\prime\in K(P, s, {\mathscr C})$, and hence, there is a Markov kernel $Q^\prime_w$ such that
$Q^\prime_w \in {\mathscr C}(w(s))$ $P$-a.s..
Because $Q_w$ is an $\eta$-maximizer,
$\eta(Q_w) \ge \eta(Q_w^\prime)$,
and then,
\[
\eta(Q) = \int \eta(Q_w)\,P(dw) \ge \int \eta(Q_w^\prime)\,P(dw) = \eta(Q^\prime)\,.
\]
This implies $Q\in V_\eta[K(P, s, {\mathscr C})]$.
Now assume $Q\in V_\eta[K(P, s, {\mathscr C})]$. Then
\be\label{six}
\eta(Q) \ge \eta(Q^\prime)\,,\quad \forall Q^\prime\in K(P, s, {\mathscr C})\,.
\ee
 There exists a Markov kernel $Q_w$ such that
$Q_w\in {\mathscr C}(w(s))$\ \  $P$-a.s. and
\[
\eta(Q) = \int \eta(Q_w)\;P(dw)\,.
\]
If $Q_w\in V_\eta[{\mathscr C}(w(s))]$\ \  $P$-a.s., we are done. If not, write $\Omega$ as $N_0\cup N_1\cup N_2$, where $P(N_0) = 0$, $Q_w\in V_\eta[{\mathscr C}(w(s))]$ for $w\in N_1$,  and
$Q_w\notin V_\eta[{\mathscr C}(w(s))]$ for $w\in N_2$. We may assume that $N_1, N_2\in {\cal F}_s$, and that  $P(N_2) > 0$.
Let $Q^{\prime}$ be any measure in $K(P, s, V_\eta[{\mathscr C}])$, and let $Q^{\prime}_w$ be the corresponding Markov kernel. We have
\[
\begin{aligned}
& \eta(Q^{\prime}) = \int \eta(Q^{\prime}_w)\;P(dw) =
\int_{N_1}\eta(Q^{\prime}_w)\;P(dw) + \int_{N_2}\eta(Q^{\prime}_w)\;P(dw) = \\
& \int_{N_1}\eta(Q_w)\;P(dw) + \int_{N_2}\eta(Q^{\prime}_w)\;P(dw) > \\
& \int_{N_1}\eta(Q_w)\;P(dw) + \int_{N_2}\eta(Q_w)\;P(dw) = \eta(Q)\,,
\end{aligned}
\]
and this contradicts \eqref{six}.
\end{proof}


\noindent$\bullet$\ {\bf Set-valued maps $\Gamma[P, x, s]({\mathscr C})$.}

Recall that $\Phi$ is a countable set of bounded continuous functions on $X$ that strongly separates points of $X$ and is closed under multiplication.
Choose a countable dense subset $\Lambda$ of $(0, +\infty)$, and let $(\lambda_n, \vp_n)$, $n = 1, 2, \dots$, be some enumeration of the Cartesian product $\Lambda\times \Phi$.
Associate with each  $(\lambda_n, \vp_n)$ the following linear continuous functional on ${\mathscr P}$:
\[
\zeta_n(P) = \int_0^\infty e^{-\lambda_n t}\int \vp_n(w(0))\,\theta_t P(dw)\,dt\,.
\]
In other words,
\be\label{zeta}
\zeta_n(P) = \int_0^\infty e^{-\lambda_n t}\int \vp_n(w(t))\, P(dw)\,dt\,.
\ee
For $s > 0$, denote
\[
\zeta_n^s(P) = \int_0^s e^{-\lambda_n t} \int \vp_n(w(t))\, P(dw)\,dt
\]
and notice that
\be
\zeta_n(P) = \zeta_n^s(P) + e^{-\lambda_n s}\,\zeta_n(\theta_s P)\,.
\ee
For $x\in X$, $P\in {\mathscr C}(x)$, and $s > 0$, define the following set
\be\label{eq:G}
\Gamma[P, x, s]({\mathscr C}) = \{Q\in {\mathscr C}(x): \theta_s Q\in K(P, s, {\mathscr C})\;\text{and}\; \zeta_n^s(Q) \ge \zeta_n^s(P), \forall n > 0\}\,.
\ee

\bigskip


\noindent$\bullet$\ {\bf Krylov's set-valued maps and abstract Markov selection theorem.}
\medskip

The definition that follows is morally similar to what Krylov in \cite{Krylov} calls Markov family of probability measures.
\begin{definition}\label{def:Krylov}
We call the set-valued map ${\mathscr C} : X \to 2^{\mathscr P}$ a Krylov map if ${\mathscr C}$
satisfies the conditions {\bf C1 - C3} of theorem \ref{thm:Markov} and, in addition,
has the following properties.
\begin{description}
\item[KP1] For every $x\in X$, and every $s > 0$, if $P\in {\mathscr C}(x)$, then $\theta_s P\in K(P, s, {\mathscr C})$.
\item[KP2] For every $x\in X$, and every $s > 0$, if $P\in {\mathscr C}(x)$, then
\[
\theta_s\left(\Gamma[P, x, s]({\mathscr C})\right) = K(P, s, {\mathscr C})\,.
\]
\end{description}
\end{definition}

\begin{definition}
A Markov selection of a Krylov set-valued map ${\mathscr C}$ is a measurable selection $x\to P_x\in {\mathscr C}(x)$ such that for any $s > 0$
\be\label{seven}
\theta_s P_x(\cdot) = \int P_{w(s)}(\cdot)\,P_x(dw)\,.
\ee
\end{definition}

Note that \eqref{seven} can be written in terms of conditional probabilities in a more traditional form
\be\label{eight}
P_x\left[\theta_s^{-1}(\cdot) |{\cal F}_s\right](w) = P_{w(s)}(\cdot),\quad (P_x,{\cal F}_s)-\text{a.s.}
\ee

\noindent Here is the abstract Markov selection theorem.
\begin{theorem}\label{Krylov}
Every Krylov set-valued map has a Markov selection.
\end{theorem}

\subsection{Proof of the abstract theorem}\label{sec:abstrMark}

\begin{lemma}\label{reduct}
Let ${\mathscr C}$ be a Krylov map. Then, for any functional $\zeta_n$, see \eqref{zeta}, the map $V_{\zeta_n}[{\mathscr C}]$ is Krylov as well.
\end{lemma}
\begin{proof}
Let $\zeta$ be one of the functionals $\zeta_n$, i.e., for $P\in{\mathscr P}$,
\[
\zeta(P) = \int_0^\infty e^{-\lambda t}\int \vp(w(0))\,\theta_t P(dw)\,dt\,,
\]
where $(\lambda, \vp) = (\lambda_n, \vp_n)$ for some $n$. As before, we denote by $\zeta^s(P)$ the
truncated integral
\[
\zeta^s(P) = \int_0^s e^{-\lambda t}\int \vp(w(0))\,\theta_t P(dw)\,dt\,.
\]
Let ${\mathscr C}$ be a Krylov map and consider
the map $V_\zeta[{\mathscr C}]: x\mapsto V_\zeta[{\mathscr C}(x)]$. To verify that this map is Krylov, we check properties {\bf C1 - C3} and {\bf KP1 - KP2} of Definition \ref{def:Krylov}. It is clear that conditions {\bf C1} and {\bf C3} are satisfied. The map $V_\zeta[{\mathscr C}]$ is measurable
by \cite[Theorem 18.19]{A-B}, hence condition {\bf C2} is verified. To check {\bf KP1}, pick any $P\in V_\zeta[{\mathscr C}(x)]$ and an $s \ge 0$. We want to show that $\theta_s P\in K(P, s, V_\zeta[{\mathscr C}])$. In view of Lemma \ref{commute}, we need $\zeta(\theta_s P)\ge \zeta(Q)$ for all $Q\in K(P, s, {\mathscr C})$. Condition {\bf KP2} for ${\mathscr C}$ implies that if  $Q\in K(P, s, {\mathscr C})$, then there exists $Q^\prime \in \Gamma[P, x, s]({\mathscr C})$ such that $\theta_s Q^\prime = Q$.  Since $Q^\prime\in {\mathscr C}(x)$ and $P$ maximizes $\zeta$ on ${\mathscr C}(x)$, we have
\[
\zeta(P) \geq \zeta(Q')\,,
\]
which we re-write as
\[
\zeta^s(P) + e^{-\lambda s} \zeta(\theta_s P) \geq \zeta^s(Q') + e^{-\lambda s} \zeta(Q)\,.
\]
On the other hand, since $Q^\prime \in \Gamma[P, x, s]({\mathscr C})$,
\[
\zeta^s(P) \leq \zeta^s(Q')\,.
\]
Combine this with the previous inequality to obtain
\[
\zeta(\theta_s P) \geq \zeta(Q)\,,
\]
which implies $\theta_s P \in V_\zeta[K(P, s,{\mathscr C})]$, and property {\bf KP1} follows.

It remains to show that
\be\label{eq:KP2}
\theta_s\left(\Gamma[P, x, s](V_\zeta[{\mathscr C}])\right) = K(P, s, V_\zeta[{\mathscr C}])
\ee
provided $P\in V_\zeta[{\mathscr C}(x)]$. Take any $Q\in K(P, s, V_\zeta[{\mathscr C}])$. Then
$Q\in K(P, s, {\mathscr C})$, and, by condition {\bf KP2} for ${\mathscr C}$, there exists
$Q^\prime\in \Gamma[P, x, s]({\mathscr C})$ such that $\theta_s Q^\prime = Q$. Since both $\theta_s P$ and $Q$ belong to $V_\zeta[K(P, s,{\mathscr C})]$, we have $\zeta(\theta_s P) = \zeta(Q) = \zeta(\theta_s Q^\prime)$. Then
\[
\zeta(Q^\prime) - \zeta(P) = \zeta^s(Q^\prime) + e^{-\lambda s} \zeta(\theta_s Q^\prime)
- \zeta^s(P) - e^{-\lambda s} \zeta(\theta_s P)
         = \zeta^s(Q^\prime) - \zeta^s(P)\,.
\]
Now recall that $Q^\prime\in \Gamma[P, x, s]({\mathscr C})$ which implies $\zeta^s(Q^\prime)\ge  \zeta^s(P)$. Thus, $\zeta(Q^\prime) \ge \zeta(P)$. Since $P$ maximizes $\zeta$ on $\mathscr C$, so does
$Q^\prime$. Then $Q^\prime\in \Gamma[P, x, s](V_\zeta[{\mathscr C}])$.
Inclusion $\subset$ in \eqref{eq:KP2} is obvious. Indeed,  let $P\in V_\zeta[{\mathscr C}(x)]$ and suppose $Q^\prime\in \Gamma[P, x, s](V_\zeta[{\mathscr C}])$. Then $\theta_s Q^\prime\in K(P, s, V_\zeta[{\mathscr C}])$ right from the definition of $\Gamma[P, x, s](V_\zeta[{\mathscr C}])$.
\end{proof}

Define now a sequence of Krylov maps recursively by setting ${\mathscr C}^0 =  {\mathscr C}$ and
${\mathscr C}^{n+1} =  V_{\zeta_{n+1}}[{\mathscr C}^n]$. For every $x\in X$, the sets ${\mathscr C}^n(x)$ form a decreasing sequences of nested convex compacta, hence their intersection,
${\mathscr C}^\infty(x)$, is a non-empty convex compact. Thus we obtain a reducted set-valued map
${\mathscr C}^\infty : x\mapsto {\mathscr C}^\infty(x)$. This map is again Krylov. We prove next that
each set ${\mathscr C}^\infty(x)$ is a singleton.

Suppose $P_1, P_2\in {\mathscr C}^\infty(x)$. Then $\zeta_n(P_1) = \zeta_n(P_2)$ for all $n$. The way
the functionals $\zeta_n$ were defined, we first conclude (by taking into account  uniqueness of the Laplace transform) that equality
\be\label{unique}
\int \vp(w(t))\,P_1(dw) = \int \vp(w(t))\,P_2(dw)\,,\quad \forall t \ge 0,
\ee
is satisfied for all functions $\vp\in\Phi$.  Since functions in $\Phi$ strongly separate points in $X$ and $\Phi$ is closed under multiplication, by Theorem~11 of \cite{B-K} the family $\Phi$ separates Borel measures on $X$.  As a consequence,
the probability measures
\be\label{trans-prob}
p_s(x, \cdot) = P(\pi_s^{-1}(\cdot))
\ee
on $X$ are independent of the choice of probability measure $P\in {\mathscr C}^\infty(x)$. Take one such $P$.
An integral form of equation \eqref{trans-prob} is
\be\label{int-tp}
\int_X \vp(y)\,p_s(x, dy) = \int_\Omega \vp(w(s))\,P(dw)\,,\quad \forall \vp\in C_b(X)\,.
\ee
By property {\bf KP1}, $\theta_s P \in K(P, s, {\mathscr C}^\infty)$. Then Strassen's theorem guarantees that there exists a Markov kernel $Q_w$ such that
\[
\int g(w)\,\theta_s P(dw) = \int \int g(v)\,Q_w(dv)\;P(dw)
\]
for any $g\in C_b(\Omega)$, and $Q_w\in {\mathscr C}^\infty(w(s))$\ \ $P$-a.s., and the map
$w\mapsto  \int g(v)\,Q_w(dv)$ is ${\cal F}_s$-measurable. Take $g(w) = \vp(\pi_t(w))$,
where $\vp\in C_b(X)$ and $t\ge 0$.
Then the above equality reads
\[
\int \vp(w(t+s))\,P(dw) = \int \int \vp(\pi_t(v))\,Q_w(dv)\;P(dw)\,.
\]
Using \eqref{trans-prob} and \eqref{int-tp}, we can re-write it as follows:
\[
\int_X \vp(y)\,p_{t+s}(x, dy) = \int \int_X \vp(y)\,p_t(w(s), dy)\;P(dw) =
\int_X \int_X \vp(y)\,p_t(z, dy)\,p_s(x, dz)\,.
\]
In other words, the measures $p_t(x, \cdot)$ satisfy the Chapman-Kolmogorov equation,
\be\label{CK}
p_{t+s}(x, dy) = \int_X p_t(z, dy)\,p_s(x, dz)\,.
\ee
The fact that $\theta_s P \in K(P, s, {\mathscr C}^\infty(x))$ for all $s$ and equation \eqref{int-tp} justify the following calculation, where $\psi_1$ and $\psi_2$ are arbitrary functions from $C_b(X)$:
\[
\begin{aligned}
& \int \psi_1(w(s)) \psi_2(w(s + t))\,P(dw) = \int \psi_1(w(0)) \psi_2(w(t))\,(\theta_s P)(dw) = \\
& \int \left(\int \psi_1(v(0)) \psi_2(v(t))\,Q_w(dv)\right)\;P(dw) =
\int \psi_1(w(s))\,\left(\int \psi_2(v(t))\,Q_w(dv)\right)\;P(dw) = \\
& \int \psi_1(w(s))\, \left(\int_X \psi_2(x_2)\,p_t(w(s), dx_2)\right)\;P(dw) = \\
&\int_X \psi_1(x_1)\, \left(\int_X \psi_2(x_2)\,p_t(x_1, dx_2)\right)\,p_s(x, dx_1)\,.
\end{aligned}
\]
More generally, we obtain
\[
\begin{aligned}
& \int_\Omega \psi_1(w(t_1)) \psi_2(w(t_2))\cdots \psi_n(w(t_n))\,P(dw) = \\
& \int_{X^n} \psi_1(x_1) \psi_2(x_2)\cdots \psi_n(x_n)\,p_{t_1}(x, dx_1)\,p_{t_2-t_1}(x_1, dx_2) \cdots p_{t_n-t_{n-1}}(x_{n-1}, dx_n)\,.
\end{aligned}
\]
This shows that the values of the measure $P$ on the sets \eqref{basic-sets}  are independent of the choice of $P$ from ${\mathscr C}^\infty(x)$. Hence, ${\mathscr C}^\infty(x)$ is a singleton.
The family of measures $P_x\in {\mathscr C}^\infty(x)$, $x\in X$, satisfies \eqref{seven}, hence it is Markov.
This completes the proof of Theorem \ref{Krylov}.


\subsection{Proof of the Markov selection theorem}\label{proof:Markov}

We are going to show that the Markov selection Theorem~\ref{thm:Markov} follows immediately from
our abstract Theorem \ref{Krylov}. To this end we check that conditions  {\bf MP1 - MP2} in the statement of theorem \ref{thm:Markov} guarantee that conditions {\bf KP1 - KP2}
for the map ${\mathscr C}$ to be Krylov are satisfied. To verify {\bf KP1}, we use assumption {\bf MP1}.
 Pick $x \in X$ and a $P \in \mathscr{C}(x)$.  From the definition of conditional probability,
 $$P(A \cap \theta_s^{-1}(B)) = \int_A P[\theta_s^{-1}(B) | \mathcal{F}_s](w) P(dw)\,$$
 for $A \in \mathcal{F}_s$ and $B \in \mathcal{F}$. If $A = \Omega$,
 \[
 \theta_s P(B) = \int P[\theta_s^{-1}(B) | \mathcal{F}_s](w) P(dw)\,.
 \]
Assumption \textbf{MP1} and Strassen's theorem then show that $\theta_s P \in K(P, s, \mathscr{C})$, hence property \textbf{KP1} is satisfied.

  Notice here that we need in the construction of $K(P, s, \mathscr{C})$ to use $P$ restricted to $\mathcal{F}_s$ and not just to $\sigma(\pi_s)$, because $w \mapsto P[\theta_s^{-1}(B) | \mathcal{F}_s](w)$ is $\mathcal{F}_s$-measurable but not $\sigma(\pi_s)$-measurable.

  Finally, with $P \in \mathscr{C}(x)$, any element in $K(P, s,\mathscr{C})$ comes from some measurable selection $Q_{\cdot}$ of the set-valued map $ w\mapsto \mathscr{C}(w(s))$.
  According to \textbf{MP2}, $P \otimes_s Q \in \mathscr{C}(x)$, and
  $$\theta_s(P \otimes_s Q)(\cdot) = \int Q_w(\cdot) P(dw)\,.$$
 Therefore, $\theta_s (P \otimes_s Q) \in K(P, s, \mathscr{C})$.
Since $P(A) = (P\otimes_s\!Q)(A)$ for any $A \in \mathcal{F}_s$, we have $\zeta^s_n[P] = \zeta^s_n[P \otimes_s Q]$. Hence, $P \otimes_s Q \in \Gamma[P, x, s]({\mathscr C})$.
 Thus, $K(P, s, \mathscr{C}) =  \theta_s (\Gamma[P, x, s]({\mathscr C}))$ and property \textbf{KP2} is satisfied. End of proof.

\bigskip

\begin{remark}
As we have just seen, assumption {\bf MP2} implies the validity of {\bf KP2}. If we try to go in the opposite direction
and assume property {\bf KP2}, there seems to be no justification for $P\otimes_s Q$ to belong to the set $\mathscr{C}(x)$ unless some additional assumptions on $\mathscr{C}(x)$ are made.
\end{remark}

\subsection{Selections with strong Markov property}

The abstract Theorem \ref{thm:Markov} can be easily extended to yield selections with strong Markov
property. We use the old notation plus a few modifications needed to include stopping times in ${\cal T}$.

Recall, that a stopping time with respect to the filtration $(\mathcal{F}_s)_{s \geq 0}$ is a map
$\tau: \Omega \to [0, \infty)$ such that  $\{w \in \Omega : \tau(w) \leq s \} \in \mathcal{F}_s$ for any $s \geq 0$. If $\tau$ is a stopping time, then
\begin{enumerate}
\item[--\ ] $\pi_\tau$ is the map $\Omega\to X$ such that $\pi_{\tau}(w) = \pi_{\tau(w)}(w)$;
\item[--\ ] $\theta_\tau$ maps $w(\cdot)\in \Omega$ to $w(\tau(w) + \cdot)$;
\item[--\ ] ${\cal F}_\tau$ is the $\sigma$-algebra made of the sets $A\in {\cal F}$ such that
  $A\cap \{w \in \Omega :\tau(w) \leq s\} \in \mathcal{F}_s$, for all $ s \geq 0$;
\item[--\ ] the construction of $P\otimes_\tau Q$ is analogous to that in Lemma~\ref{PxQ}, see, e.g.,
  \cite[Theorem 6.1.2.]{S-V} or \cite[Theorem 1.53]{klenke2013probability};
\item[--\ ] if $\zeta$ is one of the $\zeta_n$ functionals corresponding to $(\lambda, \vp) = (\lambda_n, \vp_n)$, and if $\tau$ is a stopping time, by $\zeta^\tau$ we understand the functional
\[
 P \mapsto \zeta^\tau(P) = \int_\Omega \int_0^{\tau(w)} e^{-\lambda t} f(\theta_t w) dt P(dw)\,;
\]
similarly, we have
\[
\zeta(P) = \zeta^\tau(P) + e^{-\lambda \tau} \zeta(\theta_\tau P)\,,
\]
where, by definition,
\[
e^{-\lambda \tau} \zeta(\theta_\tau P) = \int_\Omega e^{-\lambda \tau(w)}\int_0^\infty e^{-\lambda t} f(\theta_{\tau(w) + t}  w) dt P(dw)\,.
\]
\end{enumerate}

With these modifications, the sets $K(P, \tau, \mathscr{C})$ and $\Gamma[P, x, \tau](\mathscr C)$ are defined by replacing $s$ with $\tau$ in \eqref{eq:K} and \eqref{eq:G}, i.e.,
\[
K(P, \tau, \mathscr{C}) =\left\{Q \in \mathscr{P} : \int_\Omega f(x) Q(dx) \leq \int_\Omega h_{\pi_{\tau}(w)}(f) P(dw), \forall f \in C_b(\Omega) \right\}\,.
\]
and
\[
\Gamma[P, x, \tau]({\mathscr C}) = \{Q\in {\mathscr C}(x): \theta_\tau Q\in K(P, \tau, {\mathscr C})\;\text{and}\; \zeta_n^\tau(Q) \ge \zeta_n^\tau(P), \forall n > 0\}\,.
\]
Now that all these formulas work with stopping times deterministic or not, we extend the
Definition~\ref{def:Krylov}
of Krylov maps by allowing $s$ to be any positive stopping time in the properties {\bf KP1} and {\bf KP2}. We call such
Krylov maps {\it strong}.
Similarly, we extend formulas \eqref{seven} and \eqref{eight} in the definition of a Markov selection to
allow $s$ to be any positive stopping time. This gives definition of a {\it strong} Markov selection.
By repeating the arguments in Section~\ref{sec:abstrMark} with general stopping time, we obtain the following theorem.
\begin{theorem}\label{thm:strong-Krylov}
Every strong Krylov set-valued map has a strong Markov selection.
\end{theorem}

As a corollary, we have
\begin{theorem}\label{thm:strong-Markov}
Assume that the set-valued map ${\mathscr C}$ satisfies the conditions {\bf C1}-{\bf C3} of Theorem \ref{thm:Markov}.
In addition, assume that the conditions {\bf MP1} and {\bf MP2} are satisfied for any non-negative stopping time $s$.
Then there exists a measurable selection $P_x\in{\mathscr C}(x)$, $x\in X$, with the strong Markov property:
\[
P_x[\theta_\tau^{-1}(\cdot) | \mathcal{F}_\tau](w) = P_{\pi_\tau(w)} \quad (P_x, {\cal F}_\tau)\text{-a.s.}\,,
\]
for all stopping times $\tau$.
\end{theorem}


\section{Examples}\label{sec:examples}

\subsection{ODEs and differential inclusions}

There is a large literature on ordinary differential equations
\be\label{eq:ode}
\frac{du}{dt} = f(u)
\ee
lacking uniqueness: for all or some initial conditions $x$, the set of corresponding solutions (the integral funnel) $S(x)$ is not a single trajectory.
Since the work of Kneser \cite{Kn} and Fukuhara (Hukuhara) \cite{Fu} it is known that, if  $f$ is a continuous map from $\R^d$ into itself, the cross-section of $S(x)$ at time $t$, i.e., the
set $S(t, x) = \{u(t, x): u\in S(x)\}$, is compact and connected.  Also, and this goes back to Aronszajn
\cite{Ar}, every chunk $S([0, T], x)$ of the funnel $S(x)$ is the intersection of a decreasing sequence of compact, contractible sets (and, in particular, compact). This information is essentially enough
to apply our Theorem~\ref{thm:main}.

Thus, assume that $f: \R^d\to \R^d$ is continuous and such that
for every initial condition $x\in\R^d$, every solution $u(t, x)$ exists on the time interval $[0, +\infty)$, and we have
\be\label{solns}
u(t, x) = x + \int_0^t f(u(s, x))\,ds\,.
\ee
The solutions $u(\cdot, x)$ starting at $x$ are viewed as elements of the space $\Omega$ of continuous maps from $[0, +\infty)$ to $\R^d$. This space we equip with the topology of compact convergence. It is separable and metrizable; one possible metric is given by formula \eqref{met-Om} with $\rho$ being the usual metric on $\R^d$.

Denote by $S(x)$ the set of all solutions starting at $x$, and denote by $S([0, T], x)$ the restrictions of solutions $u\in S(x)$ to the time interval $[0, T]$. If $K$ is a subset of $\R^d$, $S(K) = \cup_{x\in K}S(x)$, and $S([0, T], K)$ is defined similarly.
While $S(K)$ is a subset of $\Omega$,
$S([0, T], K)$  is a subset of $C([0, T]\to \R^d)$ with the usual topology of uniform convergence.
As was shown by Aronszajn in \cite{Ar}, if $K\subset \R^d$ is compact, the set $S([0, T], K)$  is compact as well. This implies that $S(K)$ is compact in $\Omega$ for every compact set $K\subset \R^d$, verifying property
{\bf S1} of Section~\ref{sec:semi} for the map $x\to{S}(x) = S(x)$. A simple calculation
\[
\begin{aligned}
u(s + t) &= x + \int_0^s f(u(r)) dr +  \int_s^{s+t} f(u(r)) dr \\
      & = x + \int_0^s f(u(r)) dr + \int_0^t f(u(s + r)) dr \\
      & = u(s) + \int_0^t f(u(s + r)) dr \,,
\end{aligned}
\]
verifies property {\bf S3}: $\theta_s u \in S(u(s))$. Also, if $v\in S(u(s))$, then $v(r) = u(s) + \int_0^r f(v(r^\prime))\,dr^\prime$. The spliced path
\[
u\bowtie_s v (t) =
\begin{cases}
u(t) & \text{when}\;t \le s\\
v(t - s) & \text{when}\;t \ge s
\end{cases}
\]
is a solution of \eqref{eq:ode} with the initial condition $x$ because, when $t\ge s$,
\[
\begin{aligned}
& x + \int_0^t f(u\bowtie_s v (r))\,dr = x + \int_0^s f(u(r))\,dr + \int_s^t f(v(r - s))\,dr =\\
& u(s) + \int_0^{t-s}f(v(r))\,dr = v(t - s) = u\bowtie_s v (t) \,.
\end{aligned}
\]
This verifies {\bf S4}. It remains to check measurability of the map $S$, i.e., property {\bf S2}.
\begin{lemma}
  The map $x \mapsto S(x)$ is measurable in the sense that for any closed set ${\cal K} \subset \Omega$ the set $$ S^{-}({\cal K}) = \{x \in \mathbb{R}^d : S(x) \cap {\cal K} \neq \varnothing \}$$ belongs to the Borel $\sigma$-algebra of $\mathbb{R}^d$.
\end{lemma}
\begin{proof}
  We will show that $S^{-}({\cal K})$ is closed in $\mathbb{R}^d$ (i.e., $S$ is upper semi-continuous).
  Consider a convergent sequence $x_n \to x$, where $x_n \in S^{-}({\cal K})$ and, therefore, $S(x_n) \cap {\cal K}$ is not empty. For every $n$, pick an element $u_n \in S(x_n) \cap {\cal K}$.
  Since the set $S(\{x_j\}_{j\ge 1}\cup \{x\})$ is a compact in $\Omega$, and all $u_n$ belong to this set, we can extract a subsequence $u_{n_k}$ convergent to some $u$ in $\Omega$.  Passing to the limit
  in
  \[
u_n(t) = x_n + \int_0^t f(u_n(s))ds\,,
\]
we see that $u \in S(x)$. Since $u_{n_k}\in {\cal K}$ and ${\cal K}$ is closed, $u\in {\cal K}$. This
shows that
$S^{-}({\cal K})$ is closed.
\end{proof}
 We have checked all the  properties {\bf S1 - S4} for the set-valued map $x \to S(x)$. By Theorem~\ref{thm:main}, there is a measurable selection ${\frak u} : x \to {\frak u}(\cdot, x)\in S(x)$
with the semigroup property.

\bigskip

Similar arguments should and do work for differential inclusions. Consider, for example, an inclusion of the form
\be\label{eq:inclu}
\frac{du}{dt} \in f(u)\,,
\ee
where now $f$ is a set-valued map from $\R^d$ to non-empty, closed, compact subsets of $\R^d$.
Assume that $f$ is upper semi-continuous.
By definition, $u(t, x)$ is a solution of \eqref{eq:inclu} corresponding to the initial condition $u(0, x) = x$, if $u(\cdot, x)$ is an absolutely continuous function such that $u(0, x) = x$ and
\[
u(t, x) = x + \int_0^t f(u(s,  x))\,ds\,.
\]
We have not found in the literature a general result of Aronszajn type for differential inclusions
without special growth assumptions on $f$. Thus, assume that $f$ satisfies a Nagumo-type growth condition: there exists a nondecreasing continuous function $\psi:[0, +\infty)\to (0, +\infty)$ such that
\be\label{ineq:nagumo}
\int^\infty \frac{dr}{\psi(r)} = \infty\,,
\ee
and
\be
|v| \le \psi(|u|)\,,\quad\text{for all}\; v\in f(u)\,.
\ee
Then it is known (see, e.g., \cite[Theorem 2.28]{DGO}) that for every $x$ there exists at least one
solution $u(t, x)$ of \eqref{eq:inclu} defined for all $t\ge 0$.
According to \cite[Theorems 3.11 and 3.13]{DGO} (see also \cite[Theorem 7.1]{Dei}), the funnels $S(x)$ are compact. Also, as in the differential equations case, the image $S(K)$ of a compact set $K$ is compact (see \cite[Corollary 7.2]{Dei}). This is enough to conclude that under the above assumptions
the differential inclusion \eqref{eq:inclu} has solutions $x\to {\mathfrak u}(\cdot, x)$ with the semi-group property.


\subsection{On uniqueness}\label{sec:uniq}

Equations \eqref{eq:ode} with discontinuous $f(u)$ can be also treated with Theorem~\ref{thm:main}.
(A systematic study of ODEs with discontinuous right hand sides was initiated by A.~F.~Filippov, and the basic reference is  \cite{Filippov}.)
Here is a simple example that shows that the semiflows are not in general unique. But in this case
we can find the unique maximizer for a particular separating family of functions $\Phi$.

Let $X = \R$ and let $H(u)$ be a discontinuous function such that $H(u) = 1$ for $u > 0$ and $H(u) = 0$ for $u \le 0$. Consider the Cauchy problem
\be\label{H}
\frac{du}{dt} = H(u)\,,\quad u(0) = a\,.
\ee
By a solution of \eqref{H} we understand an absolutely continuous function $u(\cdot, a)$ that satisfies
\[
u(t, a) = a + \int_0^t H(u(s,a))\,ds
\]
for all $t\ge 0$. It is easy to find all
the solutions. They are
\[
u(t, a) = \begin{cases}
a + t & \text{if}\quad a > 0, \\
a & \text{if}\quad a < 0,
\end{cases}
\]
and the only non-uniqueness occurs when $a = 0$. The integral funnel $S(0)$ is formed by the paths
\[
v_c(t) = \begin{cases}
0 & \text{if}\quad t \le c ,\\
t - c & \text{if}\quad t \ge c\,
\end{cases}
\]
where $c$ runs through $[0, +\infty]$
with $v_0(t) = t$ and $v_\infty(t) = 0$ for $t\ge 0$.
It is easy to see that there are two solutions, $u_0(t, a)$ and $u_\infty(t, a)$, with the semigroup property: both solutions equal $u(t, a)$ when $a\neq 0$, but
$u_0(t, 0) = v_0(t)$ and $u_\infty(t, 0) = 0$.


Consider the following countable family, $\Phi$, of continuous functions $\vp : \R\to [0, 1]$ that separate
points of $X = \R$:
\be\label{vp-2}
\vp_{y}(x) = |x - y| \wedge 1\,,
\ee
where $y$ runs through nonzero rationals. If $0 < y < 1$, then
\[
\zeta(v_c(\cdot)) = \int_0^{+\infty} e^{-\lambda t}\,\vp_{y}(v_c(t))\,dt =
\frac{y}{\lambda} + \frac{ - 1 + 2 e^{\lambda} - e^{\lambda (1 + y)} }{\lambda ^2}\,\exp(-(c+y+1)\lambda)
\]
The expression $- 1 + 2 e^{\lambda} - e^{\lambda (1 + y)}$
is positive when $0 < \lambda \le 1$ and $0 < y < 0.489$. Thus, for such $\lambda$ and $y$, the maximum value of $\zeta(v_c)$ is attained when $c = 0$, which corresponds to $u_0(t, 0)$, and
\[
\zeta(u_0) = \frac{y}{\lambda} + \frac{ - 1 + 2 e^{\lambda} - e^{\lambda (1 + y)} }{\lambda ^2}\,\exp(-(y+1)\lambda)\,.
\]
At the same time,
\[
\zeta(u_\infty) = \int_0^{+\infty} e^{-\lambda t}\,\vp_{y}(u_\infty(t,0))\,dt = \frac{y}{\lambda}\,,
\]
and this value is smaller than $\zeta(u_0)$. Thus, if the enumeration of $\Lambda\times\Phi$ starts from small $\lambda $ and $y < 0.489$,
we pick $u_0$ as the maximizer.

However, if $0.6 < y < 1$ and $0.7 < \lambda \le 1$, the expression $- 1 + 2 e^{\lambda} - e^{\lambda (1 + y)}$ is negative. Therefore, had we started the enumeration of $\Lambda\times\Phi$ with such values of
$y$ and $\lambda$, we would have picked $u_\infty$ as the maximizer.


\subsection{PDEs}\label{sec:NS}

In the PDE setting our approach to selections with the semi-group property applies
to a large number of different equations.
We will illustrate it with the three-dimensional Navier-Stokes  equations, where uniqueness of the Hopf solutions is not known.
However, we could have chosen semilinear Schr\"odinger or wave equations, or any number of other equations or inclusions.
The following analysis relies on Theorem~\ref{thm:conditional-and-inductive}, a generalization of Theorem~\ref{thm:main}.

Consider the initial boundary value problem for the Navier-Stokes system
\be\label{N-S}
\begin{aligned}
& v_t + v\cdot \nabla v - \nu\,\Delta v + \nabla p = f\\
& \nabla\cdot v = 0
\end{aligned}
\ee
in a bounded domain $D\subset \R^3$ with smooth boundary, and the boundary conditions
\be\label{bc}
v = 0\quad\text{on}\; \partial D\,.
\ee
The exterior force $f$ is assumed to depend only on $x$.
For details on the functional set-up and the Hopf solutions see \cite{Lad} and \cite{Lad-Ver}.
We need the following function spaces.
\begin{enumerate}
\item[--\ ] ${J}_0^\infty (D)$ is the linear set of all
divergence-free $C^\infty$-smooth vectorfields
with compact supports in $D$;
\item[--\ ] ${J}_0(D)$ is the closure of the set ${J}_0^\infty (D)$ in $L^2(D)$ with the norm
$$ \|v\| = \left(\int_D |v(x)|^2\,dx \right)^{\frac12}.$$
The Hilbert space structure is inherited by ${J}_0(D)$ from $L^2(D)$.
\item[--\ ] ${J}_0^1(D)$ is the closure of the set ${J}_0^\infty (D)$ in the Dirichlet norm
\[
\|v\|_{(1)} = \left(\int_D |v_x(x)|^2\,dx\right)^{\frac12} \,.
\]
${J}_0^1(D)$ is compactly embedded in ${J}_0(D)$\,.
\end{enumerate}
Assume that $f\in L^2(D)$. We are interested in weak solutions of equations  \eqref{N-S}
(satisfying the boundary condition \eqref{bc}) corresponding to the initial conditions
$v_0\in J_0(D)$.

\begin{definition}\label{def:weak}
  A measurable vector field $v: \Dt \to \mathbb{R}^3$ is said to be a \textit{Leray-Hopf weak solution} of the Navier-Stokes equations \eqref{N-S}, \eqref{bc} with the initial condition $v_0 \in J_0(D)$ if

  \begin{enumerate}

  \item[(i)]
    $v(t)$ takes values in $J_0(D)$ for every $t \in [0, +\infty)$ with $v(0) = v_0$, and $t \mapsto v(t)$ is continuous with respect to the weak topology of $J_0(D)$;

  \item[(ii)]
    the $L^2$ norm of $v$ is uniformly bounded in time, and the $L^2$ norm of the first spatial derivative is locally square-integrable in time, i.e. $$v \in L^\infty([0, +\infty) \to J_0(D))\,\cap\,L^2_{loc}([0,\infty) \to J_0^1(D))\,;$$

  \item[(iii)]
    $v$ satisfies the integral identity
    \begin{equation}
      \label{int-ident}
      \int_0^\infty (-v, \eta_t) + ((v \cdot \nabla)\,v, \eta) + \nu\,(\nabla v, \nabla\eta) \,dt = \\
      (v_0, \eta(0)) + \int_0^\infty (f, \eta(t)) \,dt
    \end{equation}
    for every test function $\eta$ with compact support in $[0, +\infty)\times D$ and such that $\eta\in L^2_{loc}([0, +\infty)\to J_0^1(D))$ and $\eta_t\in L^2_{loc}([0, +\infty)\to J_0(D))$;

  \item[(iv)]
    $v$ satisfies the strong energy inequality \label{item:energy-inequality}
    \begin{subequations}
      \label{eq:energy-inequality}
      \begin{equation}
        \norm{v(t + T)} \leq \norm{v(t)} \wedge \frac{1}{\nu \lambda_1} \norm{f}
      \end{equation}
      \begin{equation}
        \nu \int_t^{t+T} \norm{\nabla v(s)}^2 ds \leq \frac12 \norm{v(t)}^2 + T \norm{f} \left(\norm{v(t)} \wedge \frac{1}{\nu \lambda_1} \norm{f} \right)
      \end{equation}
    \end{subequations}
    for almost every $t \in [0, +\infty)$ including $t=0$, and every $T \in [0, +\infty)$, where $\lambda_1$ is the first eigenvalue of the Stokes operator in $D$ with the Dirichlet boundary condition.
  \end{enumerate}
\end{definition}

\noindent In this definition and henceforth, $v(t)$ stands for $v(t, \cdot)$.

\begin{remark}
  The usual energy inequality is different from~(\ref{eq:energy-inequality}).
  As a priori estimates go, equation \eqref{N-S} implies, first,
  \be\label{pre-energy}
  \frac{1}{2}\,\frac{d\hfil}{dt} \|v(t)\|^2 + \nu\,\|\nabla v(t)\|^2 \le \|f\|\,\|v(t)\|\,.
  \ee
  Then, with the help of the Poincar\'e inequality
  \[
    \lambda_1\,\|v\|^2 \le \|\nabla v\|^2\,,
  \]
  proceed to
  \[
    \|v(t)\|\,\frac{d\hfil}{dt} \|v(t)\| + \nu\,\lambda_1\,\|v(t)\|^2 \le \|f\|\,\|v(t)\|
  \]
  and
  \[
    \frac{d\hfil}{dt} \|v(t)\| + \nu\,\lambda_1\,\|v(t)\| \le \|f\|\,.
  \]
  This integrates  to
  \[
    \|v(t+T)\|\le e^{- \nu \lambda_1\,T}\,\|v(t)\| + \frac{1 - e^{- \nu \lambda_1\,T}}{\nu\lambda_1}\,\|f\|\,,
  \]
  for a.e. $t \in [0, +\infty)$ and every $T \in [0, +\infty)$.
  The bound on $\norm{v(t)}$ in \eqref{eq:energy-inequality} follows. In view of \eqref{pre-energy}, boundedness of the $L^2$-norm $\|v(t)\|$ implies property (ii) in Definition~\ref{def:weak}.
  Also, integration of \eqref{pre-energy} in time from $t$ to $t + T$ leads to the bound on the Dirichlet integral of $v$ in \eqref{eq:energy-inequality}.
\end{remark}

\begin{remark}
  The strong energy inequality was first proven by Leray in \cite{Leray} for $D =  \mathbb{R}^3$, and it was not discussed by Hopf in \cite{Hopf}. In \cite{Lad}, Ladyzhenskaya included the strong energy inequality to Hopf's solutions.

  For our applications, the existence of the times where the energy inequality is not valid is the raison d'\^etre for the generalization in Section~\ref{sec:conditional-shifts}.
\end{remark}

We are going to show how the weak solutions of the Navier-Stokes equations fit into the setting of Theorem~\ref{thm:conditional-and-inductive}.
The spaces $X^i$ will be the closed balls $$X^i = \left\{u\in J_0(D): \norm{u} \leq i \right\}$$ with the weak topology of $L^2$, and $X$ will be the inductive limit of the spaces $X^i$.
As in Section~\ref{sec:inductive-limit}, $\Omega^i$ is the space $C([0,+\infty)\to X^i)$ of continuous paths in $X^i$ with the compact-open topology, and $\Omega = \varinjlim \Omega^i$.

Existence of a (global in time) weak solution (when $f\in L^2$) for every $v_0\in J_0(D)$ is well known, see \cite[Chap. 6, Sec. 6]{Lad}.
Denote by $S(v_0)$ the set of all Leray-Hopf weak solutions corresponding to the initial condition $v_0$.
Lemma~1 of \cite{Lad-Ver} proves that $S(v_0)$ is a compact subset of $\Omega$.
In fact, the energy inequality~(\ref{eq:energy-inequality}) shows that if $v_0\in X^i$, then $S(v_0)\subset \Omega^{k_i}$ (so, $k_i = i \wedge \frac{1}{\nu \lambda_1}\norm{f}$ here).
To prove measurability of the set-valued map $S$ it is sufficient to show that if $C$ is a closed set in $\Omega^{k_i}$, then the set $$F = \{u\in X^i: S(u)\cap C \neq\emptyset\}$$ is Borel in $X^i$.
We will show that $F$ is closed in $X^i$.
Assume $u_n(0)\in F$ and $u_n(0)\rightharpoonup u_\infty$ (weak convergence in $J_0(D)$); in particular, $u_\infty\in X^i$.
Let $u_n(t)$, $t\ge 0$, be the corresponding weak solutions that belong to $C$.
The energy inequalities~(\ref{eq:energy-inequality}) satisfied by $u_n(\cdot)$ show that $u_n(\cdot)$ are uniformly bounded in $L^\infty([0, +\infty)\to J_0(D))$ and $\nabla u_n(\cdot)$ have uniform bounds in $L^2([0, \ell]\times D)$, for every integer $\ell \ge 1$.
In addition, as follows from \eqref{eq:energy-inequality} and the Ladyzhenskaya's $L^4$-inequality, the integrals $$\int_0^\ell \|u_n(t)\|_{L^4}^{8/3}\,dt$$ are uniformly bounded for every $\ell$.
Thus, there exists $u\in L^\infty([0, +\infty)\to J_0(D))$ such that $\nabla u \in L^2([0, \ell]\times D)$, and $u\in L^{8/3}([0, \ell]\to L^4(D))$, for every integer $\ell \ge 1$, and there exists a subsequence $n_k$ such that $u_{n_k}(\cdot)$ converges to $u(\cdot)$ $\text{weak-}*$ in $L^\infty([0, +\infty)\to J_0(D))$ and $\nabla u_{n_k}(\cdot)$ converges weakly in $L^2([0, \ell]\times D)$ to $\nabla u(\cdot)$.
The argument in the proof of \cite[Sec. 6.6, Theorem 13]{Lad} shows that we can (and will) assume that $u_{n_k}(\cdot)$ converges strongly in $L^2([0, \ell]\times D)$ to $u(\cdot)$.
This allows one to pass to the limit in the corresponding integral identity \eqref{int-ident} and verify that $u(\cdot)$ satisfies \eqref{int-ident} with the initial condition $u_\infty$.
Clearly, inequalities \eqref{eq:energy-inequality} hold for $u(\cdot)$.
The argument in the proof of Lemma~1 in~\cite{Lad-Ver} (see formula~(9)) shows that
$u(t)$ is weakly continuous in $J_0(D)$.  Thus, $u(\cdot)\in S(u_\infty)$.

The subsequence $u_{n_k}(\cdot)$ of weak solutions converges in the space of paths $\Omega^{k_i}$ with its compact-open topology. Indeed, for every $\xi\in J_0^1(D)$, the family of scalar functions
$(u_{n_k}(\cdot), \xi)$ is uniformly bounded and equicontinuous (formula~(9) of~\cite{Lad-Ver}) on every time interval $[0, \ell]$. Clearly, the limit of $(u_{n_k}(\cdot), \xi)$ is $(u(\cdot), \xi)$.
Since $u_{n_k}(\cdot)\in C$ and $C$ is closed, $u(\cdot)\in C$. This completes the proof of measurability of $S$.

Next, we verify property {\bf S3a}.
Let $v(\cdot)$ be a weak solution corresponding to the initial condition $v_0$ such that $v(t_0 + \cdot) \in S(v(t_0))$ , i.e. the energy inequality~(\ref{eq:energy-inequality}) is valid for $t_0$, and let $w(\cdot)$ be a weak solution corresponding to the initial condition $v(t_0)$.
We have to show that the function
\begin{equation*}
  \spliced{v}{t_0}{w}(t) = \begin{cases}
    v(t) & \text{when}\; 0 \le t \le t_0 \\
    w(t-t_0) & \text{when}\; t > t_0
  \end{cases}
\end{equation*}
is a weak solution corresponding to the initial condition $v_0$.
The only problematic issue is to check the integral identity
\begin{multline}\label{vxw}
  \int_0^{t_0} (-v, \eta_t) + ((v\cdot \nabla)\,v, \eta) + \nu\,(\nabla v, \nabla\eta)\,dt \, + \\ \int_{t_0}^\infty (- {\tilde w}, \eta_t) + (({\tilde w}\cdot \nabla)\,{\tilde w}, \eta) + \nu\,(\nabla {\tilde w}, \nabla\eta) \,dt = \\ (v_0, \eta(0)) + \int_0^\infty (f, \eta(t)) \,dt\,,
\end{multline}
where ${\tilde w}(t) = w(t - t_0)$.

Let $\vartheta(t)$ be a piece-wise linear scalar function equal to $0$ when $t\le 0$, equal to $t$ when $0 < t \le 1$, and equal to $1$ when $t > 1$.
We know that
\begin{equation}
  \label{w}
  \int_{t_0}^\infty (- {\tilde w}, \eta_t) + (({\tilde w}\cdot \nabla)\,{\tilde w}, \eta) + \nu\,(\nabla {\tilde w}, \nabla\eta) \,dt = (v(t_0), \eta(t_0)) + \int_{t_0}^\infty (f, \eta(t)) \,dt\,.  
\end{equation}
On the other hand, if we replace $\eta$ in \eqref{int-ident} with $\eta_\epsilon(t, x) = (1 - \vartheta((t-s)/\epsilon)\,\eta(t, x)$, evaluate $$\int_0^\infty (- v , (\eta_\epsilon)_t) \, dt = \int_0^{t_0+\epsilon}  (1 - \vartheta((t - t_0)/\epsilon))\;(-v, \eta_t) \,dt + \frac{1}{\epsilon}\;\int_{t_0}^{t_0+\epsilon} (v, \eta)\,dt\,,$$
and pass to the limit as $\epsilon\to 0$, we obtain $$\int_0^{t_0} (- v, \eta_t) + ((v\cdot \nabla)\,v , \eta) + \nu\,(\nabla v,\nabla\eta)\,dt + (v(t_0), \eta(t_0)) = (v_0,\eta(0)) + \int_0^{t_0} (f, \eta(t))\,dt\,.$$
Add this equality to \eqref{w} and obtain \eqref{vxw}.
Thus, all the conditions of Theorem~\ref{thm:conditional-and-inductive} are satisfied and we conclude that there is a measurable selection of Leray-Hopf solutions with the semigroup property $$\U(t + T, v_0) = \U(T, \U(t, v_0))$$ for a.e. $t \in [0, +\infty)$ and every $T \in [0, +\infty)$.

\begin{theorem}
For the three-dimensional Navier-Stokes system in a bounded domain there exists a measurable
selection of Hopf solutions with the semigroup property.
\end{theorem}

\begin{remark}
  In the case of the compressible Navier-Stokes system, a selection with the semigroup property was obtained by D. Basari\'c in \cite{Basaric} using a modification of our approach in Section~\ref{sec:basic-case}.
\end{remark}

We recapitulate: the same conclusion regarding weak solutions with the semigroup property is valid for a large class of problems including but not limited to the boundary initial value problem for the semilinear wave, heat, and Schr\"odinger equations.


\end{document}